\documentclass[reqno,11pt]{amsart}
\usepackage{amssymb, amsmath}
\usepackage{graphicx}
\usepackage{cite}

\usepackage{subfig, tikz}
\usetikzlibrary{decorations.pathmorphing}
\usepackage{graphics}

\usepackage{hyperref}
\usepackage[T1]{fontenc}

\newcommand{\td}{\text{d}}

\usepackage{bm}
\usepackage{enumerate}
\theoremstyle{plain}
\newtheorem{theorem}{Theorem}[section]
\newtheorem{lemma}[theorem]{Lemma}
\newtheorem{prop}[theorem]{Proposition}

\newtheorem{conjecture}[theorem]{Conjecture}
\newtheorem{proposition}[theorem]{Proposition}
\theoremstyle{definition}
\newtheorem{remark}[theorem]{Remark}
\newtheorem{definition}[theorem]{Definition}

\newtheorem{example}[theorem]{Example}

\usepackage{soul}
\usepackage{color}
\newcommand{\red}[1]{{\color{red}{#1}}}

\numberwithin{equation}{section}

\def\be{\begin{equation}}
\def\ee{\end{equation}}
\def\bea{\begin{eqnarray}}
\def\eea{\end{eqnarray}}

\newcounter{mnotecount}[section]

\renewcommand{\themnotecount}{\thesection.\arabic{mnotecount}}

\newcommand{\mnote}[1]
{\protect{\stepcounter{mnotecount}}$^{\mbox{\footnotesize
$
\bullet$\themnotecount}}$ \marginpar{
\raggedright\tiny\em
$\!\!\!\!\!\!\,\bullet$\themnotecount: #1} }

\begin{document}
\title[]{All  toric Hermitian ALE gravitational instantons}
\author{Bernardo Araneda}
\address[Bernardo Araneda]{
School of Mathematics and Maxwell Institute for Mathematical Sciences\\ University of Edinburgh\\
King's Buildings, Edinburgh, EH9 3JZ, UK} 
\email{baraneda@ed.ac.uk}
\author{James Lucietti}
\address[James Lucietti]{
School of Mathematics and Maxwell Institute for Mathematical Sciences\\ University of Edinburgh\\
King's Buildings, Edinburgh, EH9 3JZ, UK.}
\email{j.lucietti@ed.ac.uk}

\date{\today}

\begin{abstract}
We prove that the only smooth, Ricci flat,  ALE gravitational instanton with a toric Hermitian non-K\"ahler structure is the Eguchi-Hanson instanton. The proof is analogous to the classification of toric Hermitian ALF instantons by Biquard and Gauduchon, although we avoid the use of toric K\"ahler geometry and instead perform a direct global analysis of the Tod form of the metric in Weyl-Papapetrou coordinates. This supports a conjecture by Gibbons and Bando-Kasue-Nakajima which states that any Ricci flat ALE instanton is self-dual.
\end{abstract}

\maketitle

\section{Introduction}

A complete Riemannian manifold $(M, g)$ that satisfies Einstein's equation with a prescribed curvature decay is known as a gravitational instanton~\cite{Hawking:1976jb}.   Two notable asymptotics are asymptotically locally Euclidean (ALE), which approach $\mathbb{R}^4/\Gamma$ where $\Gamma$ is a finite subgroup of $SO(4)$, and asymptotically locally flat (ALF), which approach $S^1\times S^2$ or $S^3/\Gamma$ with cubic volume growth. In this note we will be interested in the case $(M, g)$ is Ricci flat and ALE.  In this context, long ago it was conjectured that a generalised positive mass theorem should hold~\cite{Gibbons:1979xn}. That is, if $(M,g)$ is complete with nonnegative scalar curvature,  the ADM mass is non-negative and vanishes iff $(M, g)$ is Ricci flat and self-dual.  LeBrun found counterexamples to this conjecture by constructing explicit scalar flat ALE manifolds with negative mass~\cite{LB}. However, the rigidity case, namely whether the mass vanishing implies $(M, g)$ is Ricci flat and self-dual, has remained open.  In fact, Ricci-flat implies the mass vanishes and hence a  simpler conjecture emerges.

\begin{conjecture}[Gibbons \cite{Gibbons}, Bando-Kasue-Nakajima \cite{BKN}]\label{conj}
Any simply-connected Ricci flat ALE instanton is hyper-K\"ahler.
\end{conjecture}

Nakajima \cite{Nakajima} proves this conjecture is true under the topological assumptions that $(M, g)$ is a spin manifold and $\Gamma \subset SU(2)$; in fact, the positive mass theorem was also established in this context.
The purpose of this note is to give some further evidence towards Conjecture \ref{conj} under certain geometric -- as opposed to topological -- assumptions described below.

It is first instructive to contrast this to the classification of asymptotically flat (AF) instantons (these are the special case of ALF where the circle bundle at infinity is trivial). On the basis of the black hole no-hair theorem it was conjectured by Gibbons, Hawking and Lapedes that the only nontrivial AF instanton is the Euclidean Kerr solution (this includes the Euclidean Schwarzschild solution)~\cite{Gibbons, Gibbons:1979xm,Lapedes:1980st}. Remarkably, Chen and Teo constructed an explicit counterexample to this conjecture~\cite{Chen:2011tc}. The Kerr and Chen-Teo instantons are toric, that is, they possess an isometric torus action.  Interestingly, it turns out that these instantons also possess special geometry: they are Hermitian non-K\"ahler ~\cite{Aksteiner:2021fae}. Indeed, it is a curious fact that all explicitly known instantons (ALE, ALF, AF) are either hyper-K\"ahler or Hermitian non-K\"ahler. In fact, under some assumptions such as specific topologies and $S^1$-symmetry, the Hermitian condition in the ALF case has been shown to follow \cite{Aksteiner:2023djq}.

In a remarkable paper, Biquard and Gauduchon (BG) classified all smooth toric Hermitian ALF instantons, revealing a finite list:  Euclidean Kerr, Chen-Teo, Taub-bolt and Taub-NUT (with orientation opposite to the hyper-K\"ahler orientation)~\cite{Biquard:2021gwj}.  This was made possible by the fact that the local form of the metric is fully determined by an axisymmetric harmonic function on an auxiliary  3d Euclidean space, which we refer to as the Tod metric~\cite{Tod:2020ual}.  BG perform a global analysis of  the Tod metric where boundary conditions for the auxiliary harmonic function are determined by the asymptotics and the requirement of smooth degeneration of the torus symmetry.  In particular, they exploit the fact that Hermitian non-K\"ahler instantons are conformally K\"ahler and use methods in toric K\"ahler geometry to derive the appropriate boundary conditions from the associated polytope. It is the combination of all these ingredients, together with the specific local form of the Tod metric, that lead to the remarkable conclusion that the number of fixed points of the torus symmetry is at most three. The case of one-fixed point corresponds to Taub-NUT, two-fixed points is Taub-bolt or Euclidean Kerr, and three-fixed points is Chen-Teo. Interestingly, a partial converse of the BG result was obtained in \cite{OSD}, which shows that, under some mild assumption on the polytope, a toric Hermitian instanton on the complement of a toric divisor must belong to the aforementioned BG classification.

Motivated by Conjecture \ref{conj}, it is natural to ask if a similar classification can be achieved for ALE instantons. In particular, an interesting question is if there exist Hermitian non-K\"ahler ALE instantons that are not hyper-K\"ahler, e.g. a toric ALE instanton with three fixed points like the Chen-Teo instanton. A natural candidate for this is the Euclidean Pleba\'nski-Demia\'nski solution which is ALE, however, it can be shown that the conical singularities on the axes of the torus symmetry can never be removed (we give an explicit proof of this in Appendix \ref{appendix:PD}). Nevertheless, a definitive answer to this question requires a classification. Specifically,  can a classification of all smooth toric Hermitian ALE instantons be obtained?  The local geometry is again given by the Tod metric and the global analysis only differs in the asymptotics, so one expects that this should be possible.  In this note, we show that the analysis of BG for ALF metrics can be easily adapted to provide the analogous classification for ALE metrics. In fact, we also take the opportunity to simplify some of their analysis and avoid the use of toric K\"ahler geometry. Instead, we work directly with the Tod form of the metric in Weyl-Papapetrou coordinates and analyse the associated rod structure defined by the degeneration of the torus symmetry.  Our main result is the following.

\begin{theorem}\label{thm:main}
Let $(M, g)$ be a four-dimensional, smooth, complete, simply-connected, Riemannian ALE manifold that is Ricci flat. If it admits a toric\footnote{We assume the torus action has fixed points and no discrete isotropy subgroups.} Hermitian non-K\"ahler structure then it is the Eguchi-Hanson instanton (with orientation opposite to the hyper-K\"ahler orientation).
\end{theorem}

Thus we obtain a uniqueness result: the only solution has two fixed points and corresponds to the well known Eguchi-Hanson instanton. In particular, this confirms that there are no undiscovered   toric Hermitian ALE instantons. Furthermore, since the Eguchi-Hanson instanton is indeed hyper-K\"ahler, this establishes Conjecture \ref{conj} within the class of toric Hermitian instantons.  
In this context, it is also important to note the recent paper \cite{Li2023} by Li, where it is proven that the only Hermitian non-K\"ahler ALE instanton with $\Gamma\subset SU(2)$ is Eguchi-Hanson. 
By contrast, while in this paper we assume toric symmetry, we make no assumptions on $\Gamma$ (cf. Remark \ref{remark:Gamma}).

It may be worth making the distinction between `hyper-K\"ahler' and `Hermitian non-K\"ahler' more explicit. Hyper-K\"ahler instantons are Ricci-flat and have parallel complex structures, and in the ALE case they were classified by Kronheimer \cite{Kronheimer, Kronheimer2}. A Hermitian non-K\"ahler structure on a Ricci-flat manifold refers to a Hermitian structure that is not parallel. A priori, this does not rule out hyper-K\"ahler: there may exist Ricci-flat metrics which have both parallel and non-parallel complex structures {\em with the same orientation}. Indeed, any Gibbons-Hawking metric admits both of these structures \cite{Araneda:2023qio}, being both hyper-K\"ahler and also strictly conformally K\"ahler with the same orientation. However, it turns out the conformal K\"ahler structure is not global, which is why such metrics do not appear in Theorem \ref{thm:main}, cf. Remarks \ref{remark:HK1} and \ref{remark:HK2} (note that the hyper-K\"ahler and global conformal K\"ahler structures of Eguchi-Hanson have opposite orientation).

Of course, it is still possible there exist non-Hermitian ALE instantons.
A natural setting in which to investigate the existence of gravitational instantons without special geometry is Ricci flat toric manifolds.
It turns out the Einstein equations for four-dimensional toric manifolds are equivalent to a harmonic map with a negative curvature target space  defined on a fixed Euclidean plane with prescribed boundary conditions, see e.g. ~\cite{Kunduri:2021xiv}. This is essentially the same structure that occurs for vacuum stationary-axisymmetric spacetimes in General Relativity and which allows one prove the black hole no-hair theorem. Motivated by this, a systematic study of such generic Ricci flat toric instantons has been initiated~\cite{Kunduri:2021xiv}. In particular, as in the black hole case, these can be classified in terms of certain data known as the rod structure that encodes the fixed points and degeneration of the torus action~\cite{Chen:2010zu}. In particular, for AF instantons it has been shown that there exists a unique solution for any given rod structure, however, it may possess conical singularities where the torus action degenerates~\cite{Kunduri:2021xiv}. Moreover, it has been shown that topological constraints limit the possible rod structures~\cite{Nilsson:2023ina}. Nevertheless, it remains an outstanding open problem to understand what rod structures support smooth instantons. Recently, Li and Sun have made remarkable progress on this front and shown that there exists an infinite class of new Ricci-flat non-Hermitian toric AF instantons for which the conical singularities can be removed~\cite{LiSun}. It is natural to wonder if there are analogous new non-Hermitian ALE or ALF instantons.

\subsection*{Acknowledgements}
BA is supported via the ERC Consolidator/UKRI Frontier grant TwistorQFT EP/Z000157/1. We thank Maciej Dunajski and Paul Tod for comments on a draft of this paper.

\section{Local form of toric Hermitian instantons}

Let $(M,g)$ be a four-dimensional, Ricci-flat Riemannian manifold with a Hermitian structure, i.e. a compatible integrable almost-complex structure $J$. The fundamental 2-form $\omega_{ab}= J^c_{~a} g_{cb}$ induces an orientation $\bm\varepsilon = \omega\wedge\omega$, with an associated Hodge star operator $\star_{g}$. 

Let $\mathcal{W}$ denote the Weyl tensor of $g$. Viewed as an operator on 2-forms, $\mathcal{W}$ splits into self-dual (SD) and anti-self-dual (ASD) pieces, $\mathcal{W}^{\pm}=\frac{1}{2}(\mathcal{W}\pm \star_{g}\mathcal{W})$. 
The Riemannian version of the Goldberg-Sachs theorem (see e.g. \cite{Apostolov}) implies that $\mathcal{W}$ is one-sided type D, meaning that $\mathcal{W}^{-}$ has (generically) three distinct eigenvalues and $\mathcal{W}^{+}$ has a simple eigenvalue $\lambda$ and a double eigenvalue $-\lambda/2$ (where we allow $\lambda=0$). We then distinguish two cases:
\begin{itemize}
\item If $\mathcal{W}^{+}\neq0$, the manifold is strictly conformally K\"ahler (cf. \cite[Remark 4]{Derdzinski}): there is a non-constant scalar field $\Omega>0$ such that $(M, \hat{g}, {J})$, where $\hat{g}=\Omega^2g$, is a K\"ahler manifold. The complex structure $J$ is Hermitian non-K\"ahler with respect to $g$.
\item If $\mathcal{W}^{+}=0$, the manifold is (locally) hyper-K\"ahler. The complex structure $J$ may or may not be K\"ahler with respect to $g$.
\end{itemize}

We will be interested in the case $\mathcal{W}^{+}\neq0$ (see Remark \ref{remark:HK1} below), where the metric is strictly conformally K\"ahler with a conformal factor $\Omega$. Define $\xi_{a}=J^{b}{}_{a}\partial_{b}\Omega^{-1}$. Then $\xi_{a}$ satisfies (cf. \cite{Araneda:2023xnv}) $\nabla_{(a}\xi_{b)}=0$, $\xi^a\partial_a \Omega=0$ and $\iota_\xi \hat{\omega}= -\td \Omega$ (where $\hat\omega_{ab}= J^c_{~a} \hat g_{cb}$, and note the vector field $\xi$ always denotes $\xi_a$ with index raised by $g$), thus $\xi^a=g^{ab}\xi_b$ is a Killing vector field of both $(M, g)$ and $(M,\hat{g})$, and a Hamiltonian vector field of $(M, \hat{\omega})$. Moreover, Tod showed in \cite{Tod:2020ual} that the metric and fundamental 2-form can be written in coordinates $(\tau, z, x, y)$ as 
\begin{equation}\label{Todametric} 
\begin{aligned}
{g} &= {W}^{-1} ( \td \tau+A)^2+ {W}( \td {z}^2+ e^{{u}}( \td x^2+ \td y^2) ), \\
\omega & = (\td \tau+ A)\wedge\td z  + {W}e^{{u}} \td x \wedge \td y,
\end{aligned}
\end{equation}
where $\xi=\partial_{\tau}$, $W^{-1}=|\xi|^2_{g}$, and $z=\Omega^{-1}$. The variables $(u,W,A)$ satisfy the system
\begin{subequations}\label{Todasystem}
\begin{align}
 \td{A} ={}& - z^2 \partial_z (z^{-2} W e^u) \td x \wedge \td y - W_x \td y\wedge \td z + W_y \td x\wedge \td z, \label{monopoleA} \\
 0 ={}& W_{xx} + W_{yy} + \partial_{z}(z^2 \partial_z (z^{-2} W e^u)), \label{monopoleW} \\
 0 ={}& u_{xx}+u_{yy}+(e^{u})_{zz}, \label{eq:toda}
\end{align}
\end{subequations}
where \eqref{monopoleW} follows from $\td^2A=0$, and \eqref{eq:toda} is called the $SU(\infty)$ Toda equation. 

\begin{remark}\label{remark:scaling}
We note the following scaling freedom: under $(\tau, z, e^u)\to (\alpha \tau, \alpha z, \alpha^2 e^u)$ with $(x, y, W)$ fixed, where $\alpha\neq0$ is a constant, the metric and fundamental 2-form transform as $g\to\alpha^2 g$, $\omega\to\alpha^2 \omega$.
\end{remark}

The Ricci-flat condition on the metric \eqref{Todametric} implies that the function $W$ and the simple eigenvalue $\lambda$ of $\mathcal{W}^{+}$ are given by the following equations (cf. \cite{Araneda:2023xnv}):
\begin{equation}\label{eq:Weq0}
\begin{aligned}
\frac{W_0}{W} ={}& c, \qquad W_0 := z \left( 1- \frac{z u_z}{2} \right),  \\
 \lambda ={}& -\frac{2c}{z^3}
\end{aligned}
\end{equation}
where $c$ is a constant. If $c\neq0$, we can solve \eqref{eq:Weq0} to get
\begin{align}\label{eq:Weq}
W= c^{-1}  z \left( 1- \frac{z u_z}{2} \right), 
\qquad \Omega=\left(-\frac{\lambda}{2c} \right)^{1/3},
\end{align}
where the second equation follows from $z=\Omega^{-1}$.

\begin{remark}\label{remark:HK1} 
If $c=0$, we must have $W_0=0$, which we can solve for $u$. 
However we also see that $\lambda=0$, which implies $\mathcal{W}^{+}=0$, so the metric is (locally) hyper-K\"ahler.
Hyper-K\"ahler metrics which are also strictly conformally K\"ahler with the same orientation are necessarily Gibbons-Hawking (cf.\cite{Araneda:2023qio}): there exist coordinates $(x_1,x_2,x_3)$ (functions of $(x,y,z)$) such that \eqref{Todametric} becomes
\begin{subequations}
\begin{align}
 g ={}& W^{-1}(\td\tau+A)^{2} + W(\td x_1^2 + \td x_2^2 + \td x_3^2), \label{GH} \\
 \omega ={}& r^{-1}(x_1 \omega_1 + x_2 \omega_2 + x_3 \omega_3), \label{GHomega}
\end{align}
\end{subequations}
where $\omega_{1},\omega_{2},\omega_{3}$ are self-dual 2-forms generating the hyper-K\"ahler structure, and $r$ is defined by $r=\sqrt{x_1^2+x_2^2+x_3^2}$. In fact, $r=z=\Omega^{-1}$, so $r>0$. But the geometry with the point $r=0$ removed is not complete, at least for ALE/ALF instantons in this class. 
Since we are interested in gravitational instantons, which are by definition complete, this implies that hyper-K\"ahler metrics which are also strictly conformally K\"ahler {\em with the same orientation} (such as multi-centred solutions) are not allowed. See also Remark \ref{remark:HK2} below. 
\end{remark}

We will be interested in the case that $(M, g)$ is a toric manifold, that is, there is an (effective) action of the torus $T$ on $M$ that preserves $g$.  Thus we need to introduce a compatibility condition between the toric symmetry and the Hermitian structure.

\begin{definition}
$(M, g, J)$ is Hermitian toric if is Hermitian and $(M,g)$ is toric and the associated conformally K\"ahler structure $(M, \hat{g}, J)$ (when $\mathcal{W}^{+}\neq0$) is toric K\"ahler, that is,  the torus action is an isometry of $\hat{g}$ and Hamiltonian with respect to the K\"ahler form $\hat{\omega}$ (thus $\Omega$ and $J$ are also preserved by the torus symmetry).
\end{definition}

As we have seen, Hermiticity already guarantees the existence of one Hamiltonian Killing vector field $\xi$. Thus, the additional assumption of a compatible torus symmetry implies there is another Hamiltonian Killing vector field.  Tod shows that without loss of generality one can take this to be $\partial_y$ in the above~\cite{Tod:2020ual}.  Then $u$ satisfies the Toda equation \eqref{eq:toda} with an extra symmetry
\begin{equation}
u_{xx}+ (e^u)_{zz}=0,  \; 
\end{equation}
and, from \eqref{monopoleA}, $\td A= \td (F \td y)$ for some $F$.
This means we can use Ward's transformation~\cite{Ward:1990qt} which linearises the Toda equation.  Define functions  $\rho = e^{u/2}$ and  $\zeta$ by
\begin{equation}
2 \partial_x \zeta= \partial_z e^u, \qquad 2 \partial_z \zeta = - \partial_x u \; ,   \label{zetadef}
\end{equation}
where the existence of the latter follows from the Toda equation.  Now consider the coordinate change $(z,x) \to (\rho, \zeta)$. Computing the inverse of the Jacobian $\partial(\rho, \zeta)/\partial( z, x)$ implies
\[
\partial_\rho z = \rho \partial_\zeta x, \qquad \rho \partial_\rho x = - \partial_\zeta z  \; ,
\]
so the second relation implies the existence of a function $V(\rho, \zeta)$ such that
\begin{align}\label{def:V}
z= \frac{1}{2} \rho V_\rho, \qquad x= -\frac{1}{2} V_\zeta  \; ,
\end{align}
and substituting this into the first relation gives
\[
\frac{1}{\rho} (\rho V_\rho)_\rho + V_{\zeta\zeta}=0  \; .
\]
Thus $V$ is an axisymmetric harmonic function on $\mathbb{R}^3$ with cylindrical coordinates $(\rho, \zeta)$.  %
\begin{prop}[Tod metric \cite{Tod:2020ual}]\label{prop:Todmetric}
Any Ricci-flat Hermitian toric structure with $\mathcal{W}^+\neq 0$ can be written as
\begin{align}
g = W^{-1} ( \td \tau + F \td y)^2+ W \rho^2 \td y^2+ e^{2\nu} (\td \rho^2+ \td \zeta^2)  \label{eq:gTod} \\
\omega  = \frac{1}{2} ( \td \tau + F \td y) \wedge  \td ( \rho V_\rho) - \frac{1}{2} W \rho^2 \td (V_\zeta) \wedge \td y 
\end{align}
where 
\begin{subequations}\label{eq:functionsTod}
\begin{align}
W ={}& \frac{1}{2c} \left( \rho V_\rho+ \frac{V^2_\rho V_{\zeta\zeta}}{ V_{\zeta\zeta}^2+ V_{\zeta\rho}^2} \right), \label{WV} \\
e^{2\nu} ={}& \frac{W\rho^2}{4} (V_{\zeta\zeta}^2+ V_{\zeta\rho}^2), \label{nuV} \\
F ={}&  \frac{1}{2c} \left(  \frac{\rho V_\rho^2 V_{\rho \zeta}}{V_{\zeta\zeta}^2+ V_{\zeta \rho}^2} - V_\zeta \rho^2 - 2H \right) \label{FV}.
\end{align}
\end{subequations}
Here, $c \neq 0$ is a constant (see \eqref{eq:Weq}), $V(\rho, \zeta)$ is an axisymmetric harmonic function on $\mathbb{R}^3$ and $H(\rho, \zeta)$ is the harmonic conjugate to $V$ defined by
\begin{equation}
H_\rho = -\rho V_\zeta, \qquad H_\zeta= \rho V_\rho  \; .  \label{eq:Hdef}
\end{equation}
\end{prop}

Therefore, we see that any Ricci-flat, Hermitian toric structure with $\mathcal{W}^{+}\neq0$ is fully determined by an axisymmetric harmonic function on $\mathbb{R}^3$.  The classification of such geometries therefore reduces to a global analysis of this class of metrics.

\section{ALE instantons}

We first introduce ALE instantons in a general setting.

\begin{definition}\label{def:ALE}
 $(M, g)$ is ALE of order $\tau>0$ if:
 \begin{itemize}
 \item It has one end diffeomorphic to $\mathbb{R}\times S$ where $S=S^3/\Gamma$ and $\Gamma$ is a finite subgroup of $O(4)$ acting freely.
 \item $S$ admits a Riemannian  metric $\gamma$ with constant curvature $+1$.
 \item The metric on the end has behaviour
 \begin{equation}
 g = \td r^2+ r^2 \gamma+  h
 \end{equation}
 where $r>r_0$ for some constant $r_0>0$,  $| {\nabla}^k h|= O(r^{-\tau-k})$ as $r\to \infty$, $k\geq0$ is an integer, and the connection $\nabla$ and norm are with respect to the flat  asymptotic metric $\td r^2+ r^2 \gamma$.
 \end{itemize}
\end{definition}

\begin{remark}\label{remark:BKN}
By results from Bando-Kasue-Nakajima \cite{BKN}, if $(M, g)$ is ALE Ricci flat then $\tau = 4$.  This implies the norm of the Weyl tensor is $|\mathcal{W}| = O(r^{-6})$. 
\end{remark}

We now specialise to Hermitian instantons with $\mathcal{W}^{+}\neq0$. We will first establish that the asymptotic region corresponds to the conformal factor $\Omega\to 0$. To see this, note that the Hermitian condition with $\mathcal{W}^{+}\neq0$  implies $\mathcal{W}^{+}$ has a simple non-zero eigenvalue $\lambda$, and from remark \ref{remark:BKN} we deduce that $\lambda= O(r^{-6})$. Since we see from \eqref{eq:Weq} that  $\Omega \propto \lambda^{1/3}$, it follows that 
\begin{align}\label{Omega}
\Omega = O(r^{-2}).
\end{align}
We now establish a result that will be important in the global analysis.

\begin{prop}\label{prop:c}
If $(M, g)$ is a complete Hermitian ALE instanton, the constant $c$ in \eqref{eq:Weq} is $c<0$. 
\end{prop}

\begin{proof}
First observe that from the Toda form of the metric \eqref{Todametric} it is easily checked that the conformal factor $\Omega=z^{-1}$ satisfies
\begin{equation*}
\Delta_g \Omega= - 2 c \Omega^4
\end{equation*}
where $\Delta_g := -g^{ab} \nabla_a \nabla_b$ is the Laplacian. Hence integrating the basic identity
\[
\nabla^a ( \Omega \nabla_a \Omega) = - \Omega \Delta_g \Omega+ | \nabla \Omega |^2_g
\]
implies
\[
\int_M \left(  | \nabla \Omega |^2_g+ 2c \Omega^5 \right) \td \text{vol}_g = \tfrac{1}{2}  \text{lim}_{r\to \infty} \int_{S_r} \partial_n\Omega^2 \td A_r
\]
where $S_r \cong S$ is a surface of constant $r$ in the asymptotic end with unit normal $n$ and $\td A_r$ is the induced volume form on $S_r$.  But ALE implies $n \sim \partial_r$ and $\td A_r= O(r^3)$ and $\Omega= O(r^{-2})$ so the integral over $S_r$ is $O(r^{-2})$ which therefore vanishes in the limit $r\to \infty$.  Therefore, since $\Omega>0$ and is non-constant by assumption, this implies $c<0$.
\end{proof}

\begin{remark}\label{rem:ALF}
The above proposition is also valid for ALF instantons.  In this case one also has an asymptotic end diffeomorphic to $\mathbb{R}\times S$, although $S$ can also include $S^1\times S^2$ and the metric near infinity is modelled on $\td r^2+ r^2 \gamma+ \eta^2$, where $\eta$ is a 1-form on $S$, $\gamma$ is a $T$-invariant metric on $\text{ker}\, \eta$ extended to $TS$ such that $\gamma(T, \cdot)=0$, $T$ is a vector field on $S$ such that $\eta(T)=1$ and $\iota_T \td \eta=0$, and the fall-off is of order $\tau=1$~\cite{Biquard:2021gwj}.  Then the fall-off of the Weyl tensor is $\lambda= O(r^{-3})$ and hence $\Omega= O(r^{-1})$, which together with the fact that $n\sim \partial_r$ and $\td A_r= O(r^2)$, implies that the boundary term in the proof of Proposition \ref{prop:c} also vanishes.  The fact that $c<0$ was implicitly assumed in the original analysis of~\cite{Biquard:2021gwj}, although this was later justified~\cite{Biquard:2023gyl}.
\end{remark}

\begin{remark}\label{remark:HK2}
Note Proposition \ref{prop:c} implies the case $c=0$ is not allowed for complete manifolds. Therefore, the Hermitian structure is non-trivially one-sided type D, that is, the SD part of the Weyl tensor $\mathcal{W}^{+} \neq 0$. This means for the problem at hand we do not need to consider a global analysis of the $c=0$ case, which as noted above (see remark \ref{remark:HK1}) correspond to the Gibbons-Hawking metrics. Of course, it is well known that the Gibbons-Hawking metrics do extend to complete ALE hyper-K\"ahler metrics (and also ALF), the reason they do not appear in our classification is that they are not globally Hermitian non-K\"ahler.
\end{remark}

We now introduce toric instantons in the ALE setting.

\begin{definition}
We say that $(M,g)$ is toric ALE, if it is toric and ALE,  and the torus action on the end is a torus action on $S$ that preserves $\gamma$. 
\end{definition}

\begin{remark}\label{remark:Gamma}
The toric condition implies that $\Gamma$ must be a cyclic group $\mathbb{Z}_p$, and hence $S=S^3/\Gamma$ must be a lens space $L(p,q)$.
\end{remark}

We are now ready to specialise to toric Hermitian instantons. We will need a more detailed form of the asymptotic behaviour of $\Omega$ than that given by \eqref{Omega}. This can obtained by adapting the discussion in \cite[Prop. 1.5]{Biquard:2021gwj} to ALE asymptotics. 

\begin{proposition}\label{prop:falloffOmega}
Let $(M,g)$ be a toric Hermitian ALE instanton with $\mathcal{W}^{+}\neq0$. Then the conformal factor $\Omega$ and the inverse squared-norm $W$ of $\xi$ have the asymptotic behaviour
\begin{align}
\Omega ={}& \frac{k}{r^2} + O({r^{-4}}), \label{Omega2} \\
W ={}& \frac{k^2}{4 r^2} + O({r^{-4}}), \label{eq:asymptoticsW}
\end{align}
where $k\neq0$ is a constant.
\end{proposition}

\begin{proof}
We use the fact that the conformally K\"ahler condition is equivalent to the existence of a solution to the two-index twistor equation \cite[Lemma 2.1]{Dunajski:2009dqa}, and in Appendix \ref{appendix:CKY} we study this equation for ALE manifolds. This in turn is equivalent to the conformal Killing-Yano (CKY) equation. The norm squared of a CKY tensor $Z$ is $|Z|^2=4\Omega^{-2}$ \cite{Araneda:2023qio}. In Appendix \ref{appendix:CKY} we prove that $Z=Z^0+O(r^{-2})$, where $Z^0$ is a CKY tensor in the asymptotic flat metric $\td r^2+r^2 \gamma$, and that $|Z^0|=2k^{-1} r^2 (1+O(r^{-2}))$, where $k\neq 0$ is a constant. Thus $|Z|=|Z^0|(1+O(r^{-4}))$, and since $\Omega=2|Z|^{-1}$, we get
$\Omega=\frac{k}{r^2}(1+O(r^{-2}))$, so \eqref{Omega2} follows. 
To show \eqref{eq:asymptoticsW}, we use $\xi_{a}=J^{b}{}_{a}\partial_{b}\Omega^{-1}$, then $W^{-1}=|\xi|^{2}=|\td\Omega^{-1}|^{2}$ and \eqref{eq:asymptoticsW} follows from \eqref{Omega2}.
\end{proof}

\begin{remark}
The asymptotics for $\Omega$ given by \eqref{Omega2} implies that the  K\"ahler metric $\hat{g}= \Omega^2 g$ has the asymptotic behaviour
\begin{align}
\hat{g} = k^2 ( \td \hat{r}^2 + \hat{r}^2 \gamma + O(\hat{r}^4))  \; ,
\end{align}
where $\hat{r} := 1/r \to 0$ and the error term is computed with respect to $\hat{g}$. Therefore, we can compactify $(M, \hat{g})$  by adding a point corresponding to $\hat{r}=0$, although this is an orbifold point since $S=S^3/\Gamma$ for a cyclic group $\Gamma$.  This was previously noted in~\cite{Li2023}.
\end{remark}

Since $\mathcal{W}^{+}\neq0$ and we now assume toric symmetry, the Tod form of the metric (Proposition \ref{prop:Todmetric}) applies, so the geometry is encoded in an axisymmetric harmonic function $V$ on $\mathbb{R}^3$. The ALE condition allows us to deduce the asymptotic form of $V$, as follows. 

\begin{proposition}\label{prop:Vinfty}
Let $(M,g)$ be a toric Hermitian ALE instanton ($\mathcal{W}^{+}\neq0$). Let 
\begin{align}\label{def:V0}
V_0(\rho,\zeta):=2 R - \zeta  \log \left(\frac{R+\zeta}{R-\zeta}\right), \qquad R:=\sqrt{\rho^2+\zeta^2}.
\end{align}
If $\rho \geq \rho_0$ for a constant $\rho_0>0$ (i.e. strictly away from the axis),  then as  $R\to\infty$ the function $V$ in \eqref{eq:functionsTod} is
\begin{align}\label{Vinfinity}
V= V_0+ O(\log R).
\end{align}
\end{proposition}

\begin{proof}
First, rearrange the equation for $W$ in \eqref{eq:Weq} to get
\[
u_z = \frac{2}{z}- \frac{2 c W}{z^2}.
\]
From \eqref{Omega2} we see that $z^{-1}=\Omega= k/r^2+ O(r^{-4})$ and hence the asymptotic end $r\to \infty$ corresponds to $z\to \infty$. Then, using  \eqref{eq:asymptoticsW} we deduce $u_z = 2/z + O(z^{-3})$ as $z\to \infty$.
Integrating this gives $e^u = \alpha(x) z^2 \left(1 + O(z^{-2})\right)$.
Then the Toda equation implies $(\log \alpha)''+2 \alpha=0$ which has solution $\alpha = a^2 \text{sech}^2 (a x)$ for some constant $a>0$ where we have set a translation constant in $x$ to zero.\footnote{If we do not fix this constant then from \eqref{def:V} we see this adds an affine function in $\zeta$ to the asymptotics for $V$.} Without loss of generality we can set $a=1$ (by rescaling the coordinates and the metric if necessary as in Remark \ref{remark:scaling}). 
From \eqref{zetadef}, we then deduce
\begin{align}\label{rhozeta}
\rho = z\, \text{sech}(x) (1+ O(z^{-2})), \qquad
\zeta  = z \tanh (x) (1+ O(z^{-2}))  \; .
\end{align}
This gives the leading order coordinate change, and we need to invert it to find $(z,x)$. It is convenient to first define new coordinates $(R,X)$ by
\begin{align}\label{RX}
R := \sqrt{\rho^2+\zeta^2}, \qquad
X := {\rm arcsinh}(\zeta/\rho).
\end{align}
From \eqref{rhozeta} we see that $R=z+O(z^{-1})$, so the asymptotic end corresponds to $R\to \infty$, and inverting implies
\begin{align}\label{falloffz}
 z = R + O(R^{-1})
\end{align}
as $R\to \infty$.
To find $x$, we deduce from \eqref{rhozeta} that  if  $\rho\neq0$ we have $\sinh(x)=(\zeta  / \rho)(1+O(R^{-2}))$. Now, by assumption $1/\rho\leq 1/\rho_0$, so $|\zeta|/\rho \leq R/\rho_0$  and hence we can write  $\sinh(x)=\zeta/\rho+O(R^{-1})$ for this region away from the axis. Then, from the definition of $X$ in \eqref{RX}, we deduce
\begin{align}\label{falloffx}
 x = X + O(R^{-1}) .
\end{align}
Now, recall the function $V$ is defined by \eqref{def:V}. Changing this to the new coordinates \eqref{RX}, and using \eqref{falloffz}, \eqref{falloffx}, we  find
\begin{align*}
\partial_{R}V ={}& 2(1 - X\tanh(X)) + O(R^{-1}), \\
\partial_{X}V ={}& -2R(X\,{\rm sech}^{2}(X)+\tanh(X))+O(1) ,
\end{align*}
which implies
\begin{align*}
 V = 2R - 2RX\tanh(X) + O(\log R).
\end{align*}
Using \eqref{RX} we can write this in the original coordinate and obtain  \eqref{def:V0}-\eqref{Vinfinity}.
\end{proof}

\begin{remark}
The method of proof of the above proposition can also be used for toric Hermitian ALF instantons using Remark \ref{rem:ALF} and taking into account the changes in fall-offs.  This gives a more direct derivation of the asymptotics for $V$ than that presented in~\cite{Biquard:2021gwj} (which considered the problem in the conformally related toric K\"ahler geometry).
\end{remark}

\begin{example}[Eguchi-Hanson]\label{ex:EH}
With $V_0$ as in \eqref{def:V0}, consider the toric Hermitian structure determined via Proposition \ref{prop:Todmetric} by the function 
\begin{align}\label{VEH}
 V(\rho,z) = \tfrac{1}{2}V_{0}(\rho,\zeta-z_1) + \tfrac{1}{2}V_{0}(\rho,\zeta-z_2).
\end{align}
The harmonic conjugate to \eqref{VEH} is (up to an additive constant) $H = \sum_{i=1}^{2} \frac{1}{2} H_{0}(\rho,\zeta-z_{i})$, where 
\begin{align}
 H_{0}(\rho,\zeta) = \zeta R + \frac{\rho^2}{2}\log\left(\frac{R+\zeta}{R-\zeta} \right).
\end{align}
As noted in \cite{Biquard:2023rbr}, this corresponds to the Eguchi-Hanson instanton written in Tod form. To see this explicitly, introduce coordinates $(r,\theta,\phi)$ and a constant $a$ by 
\begin{align*}
 \rho^2 = \frac{(r^4-a^4)\sin^2\theta}{16}, \qquad \zeta = \frac{r^2\cos\theta}{4}, 
 \qquad y=\phi, \qquad z_1 = \frac{a^2}{4}=- z_2, \qquad c = -\frac{a^4}{16}.
\end{align*}
Replacing in \eqref{WV}, \eqref{nuV}, \eqref{FV}, the metric \eqref{eq:gTod} becomes
\begin{align*}
 g = f(r)\frac{r^2}{4}(\td\tau+\cos\theta\td\phi)^2 + \frac{\td r^2}{f(r)} + \frac{r^2}{4}(\td\theta^2+\sin^2\theta\td\phi^2), \qquad f(r)=1-(a/r)^4.
\end{align*}
\end{example}

\section{Global analysis of Hermitian toric instantons}

A toric gravitational instanton is a four-dimensional complete Riemannian manifold $(M,g)$ that is Ricci-flat and admits an (effective) isometric torus action, see e.g.~\cite{Kunduri:2021xiv}. In addition, we will assume that $M$ is simply connected and that the torus action has at least one fixed point and no points with discrete isotropy. It has been shown that under such assumptions the orbit space $\hat{M}:= M/T$, where $T=S^1\times S^1$ denotes the torus isometry group, is a 2-dimensional simply-connected manifold with boundaries and corners~\cite{Hollands:2007aj, OR}.  The Gram matrix of Killing fields is $G_{ij}:= g(\eta_i, \eta_j)$, where $\eta_i$, $i=1,2$ are the Killing vector fields generating the torus symmetry.  The interior of $\hat M$ corresponds to points in $M$ where $G$ has rank-2, the boundaries of $\hat M$ to points where $G$ is rank-1, and the corners of $\hat M$ to where $G$ is rank-0.  

The axis  $A$ is the set of points in $M$ where $G$ does not have full-rank, which corresponds to the boundaries and corners of $\hat M$.  It can be shown that Ricci flatness implies the distribution orthogonal to the span of the Killing fields $\eta_{i}$ is integrable, so one can introduce coordinates on $M \backslash A$ so that the metric takes block form
\begin{equation}
g = G_{ij} \td \phi^i \td \phi^j +q  \; ,
\end{equation}
where $\eta_i=\partial_{\phi^i}$, and  $q$ is the metric on the  orthogonal 2-surfaces which can be identified with the interior of $\hat M$.   We can define a function $\rho:= \sqrt{ \det G}\geq 0$ on $M$ which descends to a function on the orbit space $\hat M$, where $\rho>0$ corresponds to the interior of $\hat M$ and $\rho=0$ to the boundary and corners.  It is well-known that Ricci-flatness of $g$ implies that $\rho$ is a harmonic function on $(\hat M, q)$.  Therefore, one can introduce its harmonic conjugate $\zeta$ by $\td \zeta = -\star_q \td \rho$. Hence, away from any critical points of $\rho$ we can use $(\rho, \zeta)$ as a local coordinate system on the interior of $\hat M$, so  $q= e^{2\nu} (\td \rho^2+ \td \zeta^2)$ for some function $\nu$.   The system of coordinates $(\rho, \zeta, \phi^i)$ is a Riemannian analogue of the Weyl-Papapetrou coordinates that are well known in the context of stationary-axisymmetric spacetimes.  We now establish that $(\rho, \zeta)$ in fact give a global coordinate system on the interior of the orbit space.

\begin{prop}\label{prop:WP}
Let $(M, g)$ be a Ricci-flat toric ALE instanton. Then, the Weyl-Papapetrou coordinates $(\rho, \zeta)$ are a global chart on the interior of the orbit space $\hat{M}$.
\end{prop}

\begin{proof}
We can use essentially the same argument that Weinstein used in the Lorentzian setting of stationary and axisymmetric spacetimes~\cite{Weinstein}. The definition of toric ALE implies that the orbit space $(\hat M, q)$ has an asymptotic end $\mathbb{R}\times (S/T)$, where the orbit space $S/T$ of $S$ under $T$ must be homeomorphic to a closed interval and the matrix of toric Killing fields relative to $(S, \gamma)$ defined by $\gamma_{ij}:= \gamma(\eta_i, \eta_j)$ has rank-1 at the endpoints. Thus the asymptotic end of the orbit space is a strip, which in particular is simply connected.  Furthermore, from the Gram matrix and the fall-off of the metric, we deduce $\rho = r^2 \sqrt{\det \gamma_{ij} }(1  + O(r^{-\tau}))$.  Therefore, the curve $\rho=\rho_0$ for large enough $\rho_0$ must lie in the asymptotic end. In particular, this implies that for large enough  constant $\rho_0>0$  the region $\rho> \rho_0$ is simply connected.  Then, since the orbit space $\hat{M}$ is simply connected, the complement region $0 <\rho< \rho_0$ is also simply connected.  Therefore, by the uniformization theorem we can conformally map the region $0< \rho< \rho_0$ to a strip $0< \text{Im}\, w <\rho_0$ in the complex $w$-plane such that $\rho=0$ corresponds to $\text{Im}\, w=0$ and $\rho=\rho_0$ corresponds to $\text{Im}\, w= \rho_0$. Now since $\rho$ is harmonic on $(\hat{M}, q)$ it is also harmonic on the conformally related strip in the $w$-plane and by the maximum principle we must have $\rho= \text{Im}\, w$. This shows that $w=\zeta + i \rho$ is a global holomorphic coordinate on the orbit space.
\end{proof}

\begin{remark}
The same argument shows that the Weyl-Papapetrou coordinates are also global for toric ALF instantons. The only change is that $\rho = O(r)$, where $r$ is defined in Remark \ref{rem:ALF}. This generalises the argument used for toric Hermitian ALF instantons~\cite{Biquard:2021gwj}.
\end{remark}

The above proposition is significant as it means that the interior of the orbit space $\hat M$ is diffeomorphic to the half-plane $\mathbb{H}= \{  (\rho, \zeta) \: | \; \rho>0, \zeta \in \mathbb{R} \}$ and the boundary $\rho=0$ divides into intervals $(z_i, z_{i+1})$, called `rods', separated by points $z_i$, $i=1, \dots, n$, $z_1<\dots< z_n$ corresponding to the corners of $\hat M$.  It is useful to also denote $z_0=-\infty$ and $z_{n+1}= \infty$. The points $\zeta=z_i$ for $i=1, \dots, n$ are fixed points of the torus symmetry and on each rod $I_i := (z_{i}, z_{i+1})$ for $i=0, \dots, n$, the Gram matrix $G$ has rank-1 with kernel spanned by a vector $v_i$ called the rod vector.  The rod structure of a toric instanton is defined as the collection of this boundary data $\{ (I_i, v_i)_{i=0, \dots, n} \}$~\cite{Kunduri:2021xiv}.  In general the metric $g$ possesses a conical singularity at each rod $I_i$ where $v_i$ vanishes, which is absent iff $v_i$ has closed orbits of period $2\pi$ and
\begin{equation}
\frac{| \td | v_i|^2|^2}{4 | v_i |^2}  \to 1  \label{eq:noconical}
\end{equation}
as $\rho \to 0$ for $\zeta \in I_i$; furthermore, smoothness of $g$ requires that in Weyl-Papapetrou coordinates the metric components are smooth functions of $\rho^2$. This will be  important for our global analysis.

We now specialise to the case of toric Hermitian instantons as defined earlier.  The Tod metric in Proposition \ref{prop:Todmetric} is in fact already in Weyl-Papapetrou coordinates, so in particular we deduce that the coordinates $(\rho, \zeta)$ in  Proposition \ref{prop:Todmetric} are Weyl-Papapetrou coordinates (hence the notation) and hence by Proposition \ref{prop:WP} the Tod metric is valid globally. Therefore, the classification of toric Hermitian ALE instantons reduces to a global analysis of the Tod metric, which depends solely on the harmonic function $V(\rho, \zeta)$ defined on the half-plane. In order to identify the correct boundary conditions for $V$, we must examine invariants of the toric symmetry. In particular,  in the basis of Killing fields $(\partial_\tau, \partial_y)$ the Gram matrix of \eqref{eq:gTod} is
\begin{equation}
G = \left( \begin{array}{cc} W^{-1} & W^{-1} F \\ W^{-1} F & W \rho^2 + F^2 W^{-1} \end{array} \right) \; ,
\end{equation}
where $W$ and $F$ are determined in terms of the harmonic function $V$ via \eqref{WV} and \eqref{FV}.  The Gram matrix is an invariant of the instanton and hence as explained above must be a smooth function of $\rho^2$, so $W^{-1} = w_i(\zeta)+ O(\rho^2)$ near each rod $I_i$ for a smooth function $w_i$ on $I_i$  and similarly for the other components.  However, due to the complicated relationship between the Gram matrix and the harmonic function $V$, this does not seem to allow one to straightforwardly deduce the behaviour of $V$ near the rods. Fortunately, there is another invariant that allows one to do this.  

Recall that the Hermitian toric structure implies there is a conformal K\"ahler structure and that the conformal factor $\Omega = z^{-1}$, where $z$ is a coordinate in the Toda form of the metric \eqref{Todametric}, is related to the harmonic  function $V$ by \eqref{def:V}. Therefore, since $z$ is a globally defined positive function on $(M,g)$,  in Weyl-Papapetrou coordinates it must be a smooth function of $\rho^2$, which allows us to deduce the behaviour of $V(\rho, \zeta)$ near the boundary.

\begin{prop}\label{prop:Vaxis}
Let $(M, g)$ be a toric Hermitian instanton ($\mathcal{W}^+\neq 0$) with a given rod structure.  Then, as $\rho\to 0$, the function $V$  in \eqref{eq:functionsTod} is of the form
\begin{equation}
V = f(\zeta) \log\rho^2+ g(\zeta)+ O(\rho^2)    \; ,
\end{equation}
where $f(\zeta)$ and $g(\zeta)$ are respectively a piecewise linear positive function and piecewise smooth function,  with breaks at $\zeta=z_i$, $i=1, \dots n$, and the correction terms are smooth functions of $(\rho^2, \zeta)$.
\end{prop}

\begin{proof}
Smoothness of the function $z$  implies that if $\rho\to 0$ and $\zeta \in I_i$ we have
\[
z= f_i(\zeta) + O(\rho^2)  \; ,
\]
where $f_i(\zeta)$ is a smooth positive function on $I_i$ and the correction term is a smooth function of $(\rho^2, \zeta)$. Therefore, we can integrate \eqref{def:V} to obtain
\[
V = f_i(\zeta) \log\rho^2+ g_i(\zeta)+ O(\rho^2)
\]
for some smooth function $g_i$ near each $I_i$. Thus the harmonic conjugate \eqref{eq:Hdef} satisfies
\begin{equation}
H_\zeta = 2 f_i(\zeta)+ O(\rho^2), \qquad H_\rho= - f_i'(\zeta) \rho \log \rho^2 - \rho g'_i(\zeta) +O(\rho^3)  \; .  \label{eq:dHaxis}
\end{equation}
The integrability condition $H_{\zeta\rho}= H_{\rho\zeta}$ therefore implies   $f_i''(\zeta) \rho \log\rho^2= O(\rho)$ and hence $f_i''(\zeta)=0$, 
which establishes the claim.
\end{proof}

\begin{remark}
The above proposition was also derived for toric Hermitian ALF instantons~\cite{Biquard:2021gwj}. In particular, they used symplectic coordinates on the conformally related toric K\"ahler manifold $(M, \hat{g})$ and the known near axis boundary condition for the symplectic potential.  In fact, our proof above is also valid in the ALF case and gives an alternate, simpler, derivation of the near axis behaviour of $V$ directly in Weyl-Papapetrou coordinates, which avoids the use of methods of  toric K\"ahler geometry.
\end{remark}

Therefore, by Proposition \ref{prop:Vaxis} and Proposition \ref{prop:Vinfty}  we have identified the boundary conditions for $V$ near the axis and  near infinity.  Since $V$ is an axisymmetric harmonic function that is smooth at least away from the axis this is sufficient to fully fix its functional form.  The proof of this is essentially identical to the proof for the ALF case~\cite{Biquard:2021gwj}, albeit with minor changes to account for the ALE asymptotics, which we will repeat for completeness.

\begin{prop}\label{prop:fVsol}
Let $(M, g)$ be a toric Hermitian ALE instanton ($\mathcal{W}^{+}\neq0$). Then
\begin{equation}
f(\zeta) =  \sum_{i=1}^n a_i | \zeta-z_i|   \label{eq:f}
\end{equation}
and
\begin{equation}
V =\sum_{i=1}^n a_i V_0(\rho, \zeta-z_i)   \label{eq:Vsol}
\end{equation}
where $a_i$ are positive constants such that $\sum_{i=1}^n a_i=1$ and $V_0$ is defined in \eqref{def:V0}.
\end{prop}

\begin{proof}
By Proposition \ref{prop:Vaxis}, near each rod we have
\[
\rho V_\rho =2 f(\zeta) + O(\rho^2), \qquad  V_\zeta = f'(\zeta) \log \rho^2 + O(1)
\]
and 
\[
V_{\zeta \rho}  = \frac{2 f'(\zeta)}{\rho}+ O(\rho), \qquad V_{\zeta\zeta} = O(1)  \; .
\]
In particular, $V_\rho>0$ as $\rho\to 0$. Hence, since the constant $c<0$ by Proposition \ref{prop:c}, then $W= | \partial_\tau |^2\geq 0$ together with its explicit form \eqref{WV} implies that $V_{\zeta\zeta}<0$ as $\rho \to 0$. On the other hand, from Proposition \ref{prop:Vinfty} and the identity,
\[
\partial_{\zeta\zeta}V_0 =- \frac{2}{\sqrt{\rho^2+\zeta^2}}  \; ,
\]
we also see that $V_\rho>0$ and $V_{\zeta\zeta}<0$ as $\rho^2+\zeta^2\to \infty$ .   Hence, since $V_{\zeta\zeta}$ is harmonic, by the maximum principle it follows that $V_{\zeta\zeta}<0$ everywhere.  Similarly, $V_\rho>0$ everywhere by applying the maximum principle to the harmonic-like equation satisfied by $V_\rho$.  

In particular, $V_{\zeta\zeta}<0$ implies that $V$ is a concave function of $\zeta$ for all $\rho>0$ and hence taking the limit $\rho\to 0$ it follows from Proposition \ref{prop:Vaxis} that $f(\zeta)$ is a convex linear function. 
Thus the slopes on each rod $f_i':= f'(\zeta)|_{(z_i, z_{i+1})}$ are increasing $f'_i<f'_{i+1}$ for all $i=0, \dots, n$. Upon comparing to the asymptotics in Proposition \ref{prop:Vinfty} and using the fact that as $\rho\to 0$,
\[
V_0 = |\zeta | \log \rho^2+  O(1)
\]
 we deduce that the slopes on the semi-infinite rods are $f_0'=-1$ and $f_{n}'=1$. Thus we can explicitly write 
\begin{equation}
f(\zeta) =A+   \sum_{i=1}^n a_i | \zeta-z_i|   \label{eq:fA} \; ,
\end{equation}
where $A$  and $a_i:= \tfrac{1}{2} (f'_i-f'_{i-1})>0$ are constants such that $\sum_{i=0}^{n} a_i=1$.

Next, note that a harmonic function which exhibits the correct asymptotic behaviour as $\rho\to 0$ and $R\to \infty$ is given by 
\begin{equation}
V = A\log \rho^2 + \sum_{i=1}^n a_i V_0(\rho, \zeta-z_i)  \;. \label{eq:VsolA}
\end{equation}
Now suppose $\tilde{V}$ is another axisymmetric harmonic function with the same asymptotic behaviour as in Proposition \ref{prop:Vinfty} and Proposition \ref{prop:Vaxis} for some $f(\zeta)$. Then $V-\tilde{V}=O(1)$ as $\rho \to 0$ so it extends to a smooth harmonic function on $\mathbb{R}^3$. Furthermore, $V-\tilde{V}=O(\log R)$ which by harmonicity implies that $V- \tilde{V}$ must be a constant, and we may set this constant to zero since $V$ is only defined up to an affine function of $\zeta$.

We now examine the behaviour of the metric near infinity along the semi-infinite rods.  Using the above explicit forms \eqref{eq:Vsol} and \eqref{WV} one can easily check that  as $| \zeta | \to \infty$
\begin{equation}
W|_{\rho=0} = - c^{-1} A + O(|\zeta |^{-1})  \label{Winfty}
\end{equation}
However, ALE asymptotics requires that $W\to 0$ at infinity (cf. Prop. \ref{prop:falloffOmega}) and hence we must have $A=0$.
\end{proof}

\begin{remark}
In the ALF case it was shown that $f$ and $V$ are given by \eqref{eq:fA} and \eqref{eq:VsolA} respectively, where $A>0$~\cite{Biquard:2021gwj}. Thus from \eqref{Winfty} we deduce that $\xi= \partial_\tau$ has bounded norm near infinity.
\end{remark}

\section{Classification of smooth instantons}

We will now complete the classification of smooth toric Hermitian instantons. This requires us to examine the rod structure in more detail, which again follows the analysis of the ALF case closely~\cite{Biquard:2021gwj}.
\begin{lemma}\label{lem:rodv}
Let $(M, g)$ be a smooth toric ALE instanton as in Proposition \ref{prop:Vaxis}. Then, the $2\pi$-periodic rod vector $v_i$ associated to the rod $I_i$ is given by
\begin{equation}
v_i = \left\{ \begin{array}{cc}   f_i'  ( \partial_y- F_i \partial_\tau)  & \text{if $f'_i\neq 0$} \\  c^{-1} f_i^2 \partial_{\tau} & \text{if $f_i'=0$} \end{array}  \right.
\end{equation}
where $f_i:= f(z_i)$.
\end{lemma}
\begin{proof}  First we note that by integrating \eqref{eq:dHaxis} we find
\[
H = H(0, \zeta)- f'(\zeta) \rho^2 \log \rho + O(\rho^2), \qquad H(0, \zeta)= 2\int f(\zeta) \td \zeta
\]
as $\rho \to 0$, which together with Proposition \ref{prop:Vaxis} allows us to deduce the near boundary expansion of $W$ and $F$ from \eqref{WV} and \eqref{FV}.

Now assume we are on a rod $\zeta \in I_i$  such that $f_i'= f'(\zeta)\neq 0$. Then for small $\rho$, we find
\[
W = c^{-1} \left( f(\zeta)+ \frac{f(\zeta)^2 V_{\zeta\zeta}}{2f'(\zeta)^2}\right)+ O(\rho^2), \quad e^{2\nu} = W f'(\zeta)^2 + O(\rho^2)
\]
and
\begin{equation}
F= - c^{-1} \left( H(0, \zeta)- \frac{f(\zeta)^2}{f'(\zeta)} \right) + O(\rho^2) \; ,  \label{eq:Faxis}
\end{equation}
where notice the corrections $O(\rho^2)$ are all smooth functions of $\rho^2$. In particular, the $\rho^2\log\rho$ term in $F$ arising from $H$ cancels with that coming from $V_\zeta$. Next, notice that $\partial_\zeta F(0, \zeta)=0$ using $\partial_\zeta H(0, \zeta)=2 f(\zeta)$ and the fact that $f''(\zeta)=0$ on each rod, so this defines a constant on each rod
\begin{equation}
F_i:= F(0, \zeta)|_{I_i}  \; .
\end{equation}
It can be shown that $W|_{I_i}>0$; this follows from the above near axis asymptotics, the explicit form of $f(\zeta)$ and $V_{\zeta\zeta}$ in Proposition \ref{prop:fVsol}, and  the identity $\sum_{i=1}^n a_i | \zeta- z_i | \sum _{j=1}^n a_j/ |\zeta- z_j| \geq 1$. Therefore,  the Killing vector field
\[
v_i= c_i ( \partial_y- F_i \partial_{\tau})
\]
vanishes on $I_i$, where $c_i$ is a normalisation constant we will now fix.   Using the explicit form of the Tod metric \eqref{eq:gTod}, we find $| v_i |^2= c_i^2 W \rho^2 +O(\rho^4)$ and  one can show that the absence of conical singularity condition \eqref{eq:noconical} if satisfied iff  $c_i^2= f_i'^2$.  Thus, without loss of generality we can fix $c_i=f_i'$ and the rod vector is as claimed.

Now assume we are on a rod $\zeta\in I_i$ so that $f_i'=f'(\zeta)=0$. In particular, this implies $f_i=f(z_i)=f(z_{i+1})=f_{i+1}$. Then, expanding for small $\rho$ we find,
\[
W= \frac{2 f_i^2}{c g''(\zeta) \rho^2} +O(1)  \; ,  \qquad F=O(1)
\]
where note that $V_{\zeta\zeta}= g''(\zeta)+O(\rho^2)$ so $V_{\zeta\zeta}<0$ implies $g''(\zeta)<0$.  One then finds that $\partial_\tau$ vanishes at $I_i$ and the conical singularity is absent if the $2\pi$ periodic normalised rod vector is as claimed.
\end{proof}

\begin{lemma}
If $f_i' \neq 0$ and $f_{i-1}' \neq 0$, then 
\begin{equation}
F_i- F_{i-1}= \lim_{z\to z_i^+} F- \lim_{z\to z_i^-} F = c^{-1} f_i^2 \left( \frac{1}{f_i'}- \frac{1}{f_{i-1}'} \right)   \label{eq:Fdiff}
\end{equation}
where $f_i:= f(z_i)$.   

If $f_i'=0$ then
\begin{equation}
F_{i+1}- F_{i-1}= \lim_{z\to z_{i+1}^+} F- \lim_{z\to z_{i}^-} F =- c^{-1} \left(  2 f_i ( z_{i+1}-z_i) - f_i^2 \left( \frac{1}{f'_{i+1}}- \frac{1}{f'_{i-1}} \right) \right)  \label{eq:Fdiff0}
\end{equation}
\end{lemma}

\begin{proof}
The first case follows immediately from \eqref{eq:Faxis} using that $H(0, \zeta)$ is continuous at $\zeta=z_i$.   In the second case, note that $f_i=f(\zeta)=f_{i+1}$ for all $\zeta \in I_i$ and hence $\partial_\zeta H(0, \zeta)= 2 f(\zeta)= 2 f_i$ integrates to give $H(0, z_{i+1})- H(0, z_i)= 2 f_i ( z_{i+1}-z_i)$, so the result follows from \eqref{eq:Faxis}, noting that we must have $f_{i+1}'\neq 0$ and $f_{i-1}'\neq 0$.
\end{proof}

We are now ready to put all these ingredients together and complete the proof of Theorem \ref{thm:main}. For the argument that follows, we suppose $n>1$ so there are at least three rods. (The case $n=1$ will be analysed in the proof of Lemma \ref{lem:n} below.) Smoothness requires that the bases $(v_{j-1}, v_j)$ and  $(v_j, v_{j+1})$ of the  Killing fields must be related by a $GL(2, \mathbb{Z})$ transformation; this is equivalent to
\begin{equation}
v_{j-1}+ \epsilon_j v_{j+1}= l_j v_j  \label{eq:vreg}
\end{equation}
for each $j=1, \dots, n$ where $l_j\in \mathbb{Z}$ and $\epsilon_j=\pm 1$. There are three cases to consider.

Firstly, consider the case where $f_{j-1}', f_j', f_{j+1}'$ are all nonvanishing. Then by Lemma \ref{lem:rodv} the condition \eqref{eq:vreg} is equivalent to the pair of equations
\[
f_{j-1}'+\epsilon_j f_{j+1}'= l_j f'_j, \qquad f_{j-1}'F_{j-1}+ \epsilon_j f_{j+1}'F_{j+1}= l_j f_j' F_j  \; .
\]
By eliminating $f_j'$ from the second equation one finds
\[
f_{j-1}'(F_j -F_{j-1}) = \epsilon_j f_{j+1}' (F_{j+1}-F_j) 
\]
and using \eqref{eq:Fdiff} this gives
\[
f_{j+1}^2= \epsilon_j f_j^2 \frac{( f_j'- f_{j-1}')}{( f_{j+1}'- f_j')}   \; .
\]
But since $f_{j-1}'<f_j'< f_{j+1}'$ this implies $\epsilon_j=1$.  Therefore, our regularity constraints reduce to the pair of equations
\begin{equation}
f_{j-1}'+ f_{j+1}'= l_j f'_j, \qquad f_{j+1}^2=  f_j^2 \frac{( f_j'- f_{j-1}')}{( f_{j+1}'- f_j')}   \; .   \label{eq:reg}
\end{equation}

Secondly, in the case where $f_{j-1}'=0$ and $f_{j}', f_{j+1}'$ are nonvanishing, condition \eqref{eq:vreg} is equivalent to 
\[
 \epsilon_j f_{j+1}'= l_j f_j', \qquad  - c^{-1} f_{j-1}^2+  \epsilon_j f_{j+1}'F_{j+1}= l_j f_j' F_j
\]
and eliminating $f_j'$ gives $ c^{-1} f_{j-1}^2 = \epsilon_j f_{j+1}'(F_{j+1}-F_j)$ and using \eqref{eq:Fdiff} gives $f_{j-1}^2 = -\epsilon_j f_{j+1}^2 (f_{j+1}'-f_{j}')/f_{j}'$. Since $0=f_{j-1}'<f_j'<f_{j+1}'$ this implies $\epsilon_j=-1$ and hence using the fact that $f_{j-1}=f_j$, we deduce that the regularity conditions are again given by \eqref{eq:reg} with $f_{j-1}'=0$.  The same argument shows that in the case with $f_{j-1}', f_j'$ nonzero and $f_{j+1}'=0$ the regularity constraints are again given by \eqref{eq:reg} with $f_{j+1}'=0$.

The final case is $f_j'=0$ and hence $f_{j-1}'<0$ and $f_{j+1}'>0$.  The condition \eqref{eq:vreg} now gives
\[
f_{j-1}'+ \epsilon_j f_{j+1}'=0, \qquad f_{j-1}'F_{j-1}+ \epsilon_j f_{j+1}' F_{j+1}=- c^{-1} f_j^2 l_j
\]
and hence from the first we deduce $\epsilon_j=1$ and $f_{j-1}'=- f_{j+1}'$ and the second together with \eqref{eq:Fdiff0} gives
\begin{equation}
z_{j+1}- z_j = \frac{(2+l_j) f_j}{2 f_{j+1}'}  \; .   \label{eq:reg0}
\end{equation}

\begin{lemma}\label{lem:n}
Let $(M, g)$ be a smooth toric Hermitian ALE instanton ($\mathcal{W}^{+}\neq0$). Then, the number of fixed points of the toric symmetry is 
\begin{align}\label{eq:n}
 n=2.
\end{align}
\end{lemma}

\begin{remark}
The result \eqref{eq:n} is to be contrasted with the ALF case, where the BG classification \cite{Biquard:2021gwj} shows that there are smooth toric Hermitian ALF instantons for $n=1,2,3$, namely Taub-NUT for $n=1$, Kerr and Taub-bolt for $n=2$, and Chen-Teo for $n=3$. 
Our result \eqref{eq:n} shows that not only is there no ALE Hermitian analogue of Chen-Teo, but there is no ALE Hermitian analogue of Taub-NUT either. One may have thought that the $n=1$ ALE case could be flat space (written as a single-centred ALE Gibbons-Hawking solution), however, as we saw in Remarks \ref{remark:HK1} and \ref{remark:HK2}, the corresponding conformal K\"ahler structure is not global.
\end{remark}

\begin{proof}
First consider the $n=1$ case, so the rod structure is given by the two semi-infinite rods.  Then, noting $A=0$, we deduce from \eqref{eq:f}  that $f(\zeta) = | \zeta-z_1|$ and hence from \eqref{eq:Vsol} the harmonic function is  simply $V= V_0(\rho, \zeta-z_1)$. It is easily checked that this choice of harmonic function gives $W=0$ identically. Hence, this case is not allowed.

Consider then $n>1$, and three consecutive rods $I_{j-1}, I_j, I_{j+1}$. We will show that the slopes $f_{j+1}'>0$ and $f_{j-1}'<0$.  It then immediately follows that $n\leq 3$. 

Suppose $f_{j+1}' \leq 0$, so $f_{j-1}'<f'_j<0$. The first equation in \eqref{eq:reg} implies $l_j > 0$. If $l_j=1$ then the first equation says $f_{j+1}'=f_j'-f_{j-1}'>0$ which is a contradiction.  Hence $l_j\geq 2$.  On the other hand, since $f_j'<0$ implies $f_{j+1}<f_j$, the second equation in \eqref{eq:reg} implies $f_j'-f_{j-1}'< f'_{j+1}-f_j'$ which is the same as $f_{j+1}'+f_{j-1}'> 2f_j'$. Comparing to the first equation shows that $l_j<2$, which is a contradiction.  Therefore, $f_{j+1}'>0$.

Now suppose $f_{j-1}'\geq 0$, so $f_{j+1}'>f_j'>0$. The first equation then implies $l_j>0$ and if $l_j=1$ it gives $f_{j-1}'=f_j' - f_{j+1}'<0$ which is a contradiction. Hence $l_j\geq 2$.  But $f_j'>0$ implies $f_{j+1}>f_j$ so the second equation implies $f_j'-f_{j-1}'> f'_{j+1}-f_j'$ which is the same as $f_{j+1}'+f_{j-1}'<  2f_j'$. Comparing to the first equation then gives $l_j<2$, which is a contradiction.  Therefore, $f_{j-1}'<0$.

Now consider the $n=3$ case. Thus, from the above we must have $-1=f_0'<f_1'<0<f_2'<f_3'=1$. Following BG, set $p=-f_1'$ and $q= f_2'$ so $p, q\in (0,1)$. The first of the regularity conditions \eqref{eq:reg} for $j=1$ and $j=2$ read
\begin{equation}
1-q=p l_1, \qquad 1-p=q l_2  \label{eq:n3reg1}
\end{equation}
respectively, which in particular imply $l_1, l_2>0$. The second of the regularity conditions \eqref{eq:reg} for $j=1$ and $j=2$ read
\begin{equation}
f_2^2(p+q)= f_1^2(1-p), \qquad f_2^2(p+q)= f_3^2(1-q)  \; .  \label{eq:n3reg2}
\end{equation}
On the other hand, since $f_1'=f'(z)|_{I_1}<0$ it follows that $f_2=f(z_2)<f(z_1)=f_1$ and hence the first equation in \eqref{eq:n3reg2} implies $p+q>1-p$, which means $2p>1-q=p l_1$ where the final equality follows from the first equation in \eqref{eq:n3reg1}.  We deduce that $l_1<2$ and hence $l_1=1$ as it is an integer.  Similarly, $f_2'>0$ implies $f_2<f_3$ and hence the second equation in \eqref{eq:n3reg2} implies $2q>1-p=q l_2 $ where the final equality follows from the second equation in \eqref{eq:n3reg1}.   Hence, we have $l_2<2$ and therefore $l_2=1$.   Now, the $j=1$ and $j=2$ cases of the smoothness condition \eqref{eq:vreg} give
\[
v_0+v_2= v_1, \qquad v_1+v_3=v_2
\]
and eliminating say $v_1$ shows that $v_0=-v_3$.  This means that the rod vectors of the semi-infinite rods $I_0$ and $I_3$ are parallel, which is incompatible with ALE asymptotics.  More explicitly, if we choose a basis of rod vectors $v_0=(0,1)$ and $v_1=(1,0)$, then $v_2=(1, -1)$ and $v_3=(0,-1)$, which is only compatible with AF asymptotics.\footnote{This is the rod structure of the Chen-Teo instanton.} Therefore, the $n=3$ case is not possible.
\end{proof}

Let us now examine the remaining case $n=2$, which consists of three rods $I_0,  I_1, I_2$. Now Lemma \ref{lem:n}  says that $-1=f_0'<f_1'<f_2'=1$, so in particular $f_1'$ may be positive or negative or vanish. If $f_1' \neq 0$ the regularity conditions are \eqref{eq:reg} and reduce to
\[
l_1 f_1'=0 , \qquad f_2^2= f_1^2 \frac{f_1'+1}{1-f_1'}  \; ,
\]
so the first implies $l_1=0$.
But  \eqref{eq:vreg} reduces to $v_0+v_2=l_1 v_1$, which implies the rod vectors of the semi-infinite rods $v_0=-v_2$ are parallel, which is incompatible with ALE asymptotics.  Therefore, ALE implies $f_1'=0$, in which case, the regularity conditions \eqref{eq:reg0} reduce to
\begin{equation}
z_2-z_1= \frac{(2+ l_1) f_1}{2}   \; .  \label{eq:n2reg}
\end{equation}
Now, recalling that ALE also requires $A=0$, \eqref{eq:f} reduces to
\[
f(\zeta) = \tfrac{1}{2} | \zeta-z_1|+ \tfrac{1}{2} | \zeta-z_2|
\]
which in particular implies  $f_1= \tfrac{1}{2} (z_2-z_1)$. 
Combining with the regularity condition \eqref{eq:n2reg} gives $l_1=2$, so we deduce  the rod vectors are related by $v_0+v_2=2 v_1$. Hence choosing a basis $v_0=(0,1)$ and $v_1=(1,0)$ we have $v_2=( 2, -1)$, so the metric is ALE with  cross-section $S=L(2,1)$ at infinity.  The full solution in this case is given by the harmonic function
\[
V= \tfrac{1}{2} V_0(\rho, \zeta-z_1)+ \tfrac{1}{2} V_0(\rho , \zeta-z_2)
\]
As shown in Example \ref{ex:EH} this corresponds to the Eguchi-Hanson instanton. This completes the proof of Theorem \ref{thm:main}.

\appendix

\section{Regularity of the Pleba\'nski-Demia\'nski metric}
\label{appendix:PD}

Consider the Pleba\'nski-Demia\'nski (PD) metric as given in~\cite{Chen:2015vva, Tod:2024zpa}, 
\begin{equation}\label{eq:PDmetric}
g= \frac{1}{(p-q)^2}\left[ \frac{(1-p^2 q^2) \td p^2}{P}+ \frac{P(\td \phi - q^2 \td \tau)^2}{(1-p^2 q^2)} - \frac{(1-p^2 q^2) \td q^2}{Q}-  \frac{Q(\td \tau - p^2 \td \phi)^2}{(1-p^2 q^2)}  \right]
\end{equation}
where $P=F(p)$, $Q= F(q)$ and $F(x)=a_0 x^4+a_3 x^3 + a_2 x^2 +a_1 x + a_0$. Without loss of generality we can assume $a_0>0$. The metric is Ricci-flat, ALE, toric, and has a (local) Hermitian non-K\"ahler structure. The Weyl tensor is self-dual iff $a_3=a_1$, and the metric is flat iff $a_3=a_1=0$. We wish to investigate if there exist parameter choices such that \eqref{eq:PDmetric} extends to a smooth metric on a complete Riemannian manifold. 

\subsection{Coordinate ranges}\label{sec:PDcoord}
We will  assume that the coordinate ranges are such that $1-p^2 q^2>0$ and $p\neq q$  in order to avoid a curvature singularity at $p^2q^2=1$ and the conformal factor blowing up  \cite{Tod:2024zpa}. Riemannian signature requires $g_{pp}> 0$ and $g_{qq}> 0$ and hence $P>0$ and $Q<0$.  We will also assume that $F$ has four real roots $p_1<p_2<p_3<p_4$ so we can write $F(x)=a_0 (x-p_1)(x-p_2)(x-p_2)(x-p_4)$ with $p_1 p_2p_3 p_4=1$. 
Thus we must restrict $p\in I_1$ and $q\in I_2$ where $I_1$ and $I_2$ are open intervals between adjacent roots where $P>0$ and $Q<0$ respectively (so $P=0$ at the endpoints of $I_1$ and $Q=0$ at the endpoints of $I_2$).  The possibility of $I_1, I_2$ being between $p_1$ or $p_4$ and infinity is disallowed by our assumption $p^2q^2<1$.  Since $a_0>0$ there are two cases: $I_1=(p_2, p_3)$ and either $I_2=(p_1, p_2)$ or $I_2=(p_3, p_4)$.  However, by sending $(p, q)\to (-p, -q)$ these two cases are interchanged and hence without loss of generality we may assume $I_2=(p_1, p_2)$.  Therefore, we can take our coordinate range to be the rectangle $p_2< p<p_3$ and $p_1<q<p_2$ and note $p>q$. 

Before moving on we comment on the assumption that $F$ has four distinct real roots. If $F$ has repeated real roots then these will correspond to new asymptotic ends rather than fixed point sets of the toric symmetry.  If $F$ has a pair of real roots $p_1<p_2$ and a pair of complex conjugate roots then both $p, q \in (p_1, p_2)$ but then $P$ and $Q$ will have the same sign which is incompatible with Riemannian signature.  If $F$ has no real roots then $p, q \in \mathbb{R}$ and it is not possible to avoid the curvature singularity at $1-p^2 q^2=0$.

The region $p\to p_{2}^{+}$, $q\to p_{2}^{-}$ corresponds to an asymptotic end. To see this, change coordinates from $(\tau,\phi,p,q)$ to $(\psi,\varphi,r,\theta)$ defined by
\begin{align*}
\tau = \tfrac{(1+p_2^2)\psi-(1-p_2^2)\varphi}{P'(p_2)}, \qquad
\phi = \tfrac{(1+p_2^2)\psi+(1-p_2^2)\varphi}{P'(p_2)}, \qquad
p = p_2+\frac{c \cos^2(\tfrac{\theta}{2})}{2r^2}, \qquad 
q = p_2-\frac{c \sin^2(\tfrac{\theta}{2})}{2r^2},
\end{align*}
where $c$ is a non-zero constant. Then, as $r\to\infty$, a calculation shows that \eqref{eq:PDmetric} becomes
\begin{align*}
g \approx \frac{8(1-p_2^4)}{cP'(p_2)}\left[ \td{r}^2+\frac{r^2}{4}\left((\td\psi+\cos\theta\td\varphi)^2+\td\theta^2+\sin^2\theta\td\varphi^2 \right) \right].
\end{align*}
Therefore the metric \eqref{eq:PDmetric} is ALE provided the angles are suitable identified.

\subsection{Rod structure}
The two commuting Killing fields are $\partial_{\tau},\partial_{\phi}$, and the determinant of the Gram matrix $G$ is $\rho^2=-PQ/(p-q)^4$. Therefore ${\rm rank}(G)=1$ when either $P=0$ or $Q=0$, and ${\rm rank}(G)=0$ when both $P=0=Q$. Inspecting the rectangle $\{p_2< p<p_3, \, p_1<q<p_2\}$, we see that there are four rods $\mathcal{R}_1,...,\mathcal{R}_4$, and three fixed points of the torus symmetry:
\begin{itemize}
\item $\mathcal{R}_1$: $p=p_2$, $p_1<q<p_2$, rod vector $\ell_1 = \frac{2}{P'(p_2)}(p_2^2\partial_{\tau}+\partial_{\phi})$,
\item $\mathcal{R}_2$: $p_2<p<p_3$, $q=p_1$, rod vector $\ell_2 = \frac{2}{Q'(p_1)}(\partial_{\tau}+p_1^2\partial_{\phi})$,
\item $\mathcal{R}_3$: $p=p_3$, $p_1<q<p_2$, rod vector $\ell_3 = \frac{2}{P'(p_3)}(p_3^2\partial_{\tau}+\partial_{\phi})$,
\item $\mathcal{R}_4$: $p_2<p<p_3$, $q=p_2$, rod vector $\ell_4 = \frac{2}{Q'(p_2)}(\partial_{\tau}+p_2^2\partial_{\phi})$.
\end{itemize}
The rod vectors $\ell_1,...,\ell_4$ are $2\pi$-periodic, and the overall scaling in each case has been chosen to avoid a conical singularity, using the condition \eqref{eq:noconical}. In the analysis that follows, we will assume the generic situation in which consecutive rod vectors are not parallel. The case in which two of them become parallel occurs when two fixed points merge, and corresponds to the self-dual case, which will be analysed in section \ref{sec:SDPD}.

The fixed points are located at $(p,q)=(p_2,p_1),(p_3,p_1),(p_3,p_2)$. Smoothness at these points is equivalent to requiring that the $U(1)^2$-action generated by the bases $(\ell_1, \ell_2)$, $(\ell_2, \ell_3)$ and $(\ell_3, \ell_4)$ are pairwise related by a $GL(2, \mathbb{Z})$ transformation. For the pair $\{(\ell_1,\ell_2),(\ell_2, \ell_3)\}$, and for the pair $\{(\ell_2, \ell_3),(\ell_3, \ell_4)\}$, these conditions are respectively equivalent to
\begin{align}\label{eq:reg1}
 \ell_3 = - \epsilon \ell_1+ m \ell_2,  \qquad
 \ell_4 = - \bar{\epsilon} \ell_2 + n  \ell_3
\end{align}
where $\epsilon=\pm 1$, $m\in \mathbb{Z}$, and $\bar{\epsilon}=\pm 1$, $n\in \mathbb{Z}$. Using the explicit form of the rod vectors, these equations lead to
\begin{align*}
m &= \frac{Q'(p_1)}{P'(p_3)} \frac{(p_3^2-p_2^2)}{(1-p_1^2p_2^2)}, \quad 
\epsilon= -\frac{P'(p_2)}{P'(p_3)} \frac{(1-p_1^2p_3^2)}{(1-p_1^2p_2^2)}, \quad \\
n &= \frac{P'(p_3)}{Q'(p_2)} \frac{(p_2^2-p_1^2)}{(1-p_1^2p_3^2)}, \quad 
\bar{\epsilon}= -\frac{Q'(p_1)}{Q'(p_2)} \frac{(1-p_2^2p_3^2)}{(1-p_1^2p_3^2)}  \; .
\end{align*}
Recalling $1-p^2q^2>0$, and using $P'(p_2)>0$, $P'(p_3)<0$, we deduce $\epsilon>0$, so $\epsilon=1$. Similarly, $Q'(p_1)<0$, $Q'(p_2)>0$ imply $\bar\epsilon>0$ and thus $\bar\epsilon=1$.
Notice also that $n\epsilon$ and $m/(\epsilon \bar{\epsilon})$ simplify, and using $\epsilon=1=\bar\epsilon$ we deduce
\begin{align}\label{eq:mn}
m= \frac{p_3^2-p_2^2}{1-p_2^2p_3^2}, \qquad 
n= \frac{p_1^2-p_2^2}{1-p_1^2p_2^2}.
\end{align}

\subsection{Regularity constraints}
We wish to determine if given $m, n \in \mathbb{Z}$ there exist $p_1<p_2<p_3<p_4$ satisfying \eqref{eq:mn}.  The constraint $p_1p_2p_3p_4=1$ implies there are three possibilities: (i) all roots positive, (ii)  $p_1<p_2<0<p_3<p_4$, (iii) all roots negative.

In case (i), we see that $p_1^2<p_2^2<p_3^2$ and hence $m>0$ and $n<0$. We can invert \eqref{eq:mn} to get 
\begin{align*} p_2^2= \frac{p_1^2+ |n|}{ 1+ |n| p_1^2}, \qquad  p_3^2= \frac{p_2^2+m}{1+m p_2^2} \; ,
\end{align*}
so for given $n, m$ this gives $p_2, p_3$ in terms of $p_1$.  The second equation gives 
\begin{align*}
p_3^2-p_2^2= \frac{m (1-p_2^4)}{1+m p_2^2} \; ,
\end{align*}
which implies $p_2^4<1$. Hence $p_2^2<1$ and the first equation  then implies $$\frac{p_1^2+ |n|}{ 1+ |n| p_1^2}<1 \; ,$$ which is true iff  $|n|-1 < p_1^2 (|n|-1)$. Since $|n|  \in \mathbb{N}$ the latter inequality implies $|n|-1>0$  and hence $p_1^2>1$. But $p_1^2<p_2^2<1$ which is a contradiction. 

Consider now case (ii). We have $p_2^2<p_1^2$ and hence $n>0$. In particular this implies $$\frac{1-p_2^2 p_3^2}{1-p_1^2 p_3^2}>1 \; .$$ Also, $$-\frac{Q'(p_1)}{Q'(p_2)}= \left(\frac{p_3-p_1}{p_3-p_2}\right)\left( \frac{ p_4-p_1}{p_4-p_2}\right)>1\; ,$$ where the inequality follows since each factor in the round brackets is $>1$ by the ordering of the roots.  It therefore follows that $\bar{\epsilon}>1$ which is a contradiction.  Note this case could include $m>0, m=0$ or $m<0$ depending on the sign of $p_3^2-p_2^2$.

Finally, in case (iii) we have $p_1<p_2<p_3<p_4<0$ and hence $p_2^2<p_1^2$ and $p_3^2<p_2^2$ so $n>0$ and $m<0$.  Since $p_2^2<p_1^2$ this case can be ruled out as in case (ii), or by a similar argument to case (i).

To summarise we have shown that, assuming the three fixed points are distinct, one can never fully satisfy the regularity constraints for the PD metric \eqref{eq:PDmetric}, consistently with Lemma \ref{lem:n}.

\subsection{The self-dual case}\label{sec:SDPD}
As mentioned, the PD metric is self-dual iff $a_3=a_1$. Note this implies that if $p_i$ is a root then so is $1/p_i$ (recall no root can vanish since $p_1p_2p_3p_4=1$). If we let $p_1<p_2$ be the first two roots, then $p_3=1/p_i$ for some $i$. One can show that the only possibilities compatible with the ordering of the roots are: (a) $p_3=1/p_2$, $p_4=1/p_1$, and (b) $p_3=1/p_4$, $p_2=1/p_1$. We analyse these separately.

First consider case (a). This is consistent with all roots positive or all roots negative, that is, either $0<p_1<p_2<1/p_2<1/p_1$ or $p_1<p_2<1/p_2<1/p_1<0$.  If the roots are all positive then $p_1^2<p_2^2<1$ and if the roots are all negative then $p_1^2>p_2^2>1$.  In any case, note that $p_3=1/p_2$ implies that the vectors $\ell_3$ and $\ell_4$ are parallel, so the above analysis which assumes that $(\ell_3, \ell_4)$ is a basis does not apply.  However, since $p_1<p_3=1/p_2$ and both $p_1 , p_2$ have the same sign we have $p_1^2\neq 1/p_2^2$ and hence $(\ell_1, \ell_2)$ is a basis. Similarly, $p_1^2\neq 1/p_3^2=p_2^2$ since $p_1<p_2$ and they have the same sign, so $(\ell_2, \ell_3)$ is a basis.  Therefore, the first equation in \eqref{eq:reg1} still applies, but the second does not apply. 
We now have $$m= \frac{(1-p_1^2)(1+p_2^2)}{1-p_1^2 p_2^2} \; , \qquad \epsilon= \frac{p_2^2-p_1^2}{1-p_1^2p_2^2} \; ,$$ where we have used the simplifications $Q'(p_1)/P'(p_3)= p_2^2(1-p_1^2)/(1-p_2^2)$ and $-P'(p_2)/P'(p_3)= p_2^2$  (note this latter equation shows that $\ell_3=-\ell_4$).  In the case of all positive roots we have that $1>p_2^2>p_1^2$ and $1-p_1^2p_2^2>0$ so $\epsilon=1$ and $m>0$, and in the case of all negative roots we have $1<p_2^2<p_1^2$ and $1-p_1^2 p_2^2<0$ so again $\epsilon=1$ and $m>0$.  However, solving $\epsilon=1$ gives $p_2^2=1$ which is a contradiction. Hence case (a) is not allowed.

Now consider case (b).  This is only consistent with two roots negative and two positive, that is, $p_1<1/p_1<0<p_3<1/p_3$.  Thus $p_1^2>1$ and $p_3^2<1$. In this case the rod vectors $\ell_1$ and $\ell_2$ are collinear (since $p_2=1/p_1$).  Hence the first equation in \eqref{eq:reg1} does not apply since $\ell_1, \ell_2$ is not a basis.  On the other hand  $(\ell_2, \ell_3)$ is basis  iff $p_1\neq - 1/p_3$ and $(\ell_3 , \ell_4)$ is a basis iff $p_3\neq -1/p_2=-p_1$. But since $p_1<-1$ and $0<p_3<1$ the latter must always be the case, that is, $(\ell_3, \ell_4)$ is always a basis. Hence the second equation in \eqref{eq:reg1} applies iff $p_1\neq - 1/p_3$.  If $p_1=-1/p_3$ then there are only two independent rod vectors $\ell_1$ and $\ell_4$, which gives the rod structure of flat $\mathbb{R}^4$; indeed, in this case $p_1=-p_4$ and $p_2=-p_3$ so $P= a_0(p^2- p_3^2)(p^2-1/p_3^2)$, hence $a_3=a_1=0$. Thus it remains to consider the case where $p_1\neq -1/p_3$ so that the second equation in \eqref{eq:reg1} applies.  Now 
$-Q'(p_1)/Q'(p_2)=p_1^2$  (so $\ell_1=-\ell_2$) and $\bar{\epsilon}= (p_1^2-p_3^2)/(1-p_1^2p_3^2)$. But since $p_1^2>p_3^2$  this implies $\bar{\epsilon}=1$ which then gives $p_1^2=1$: a contradiction.  Hence case (b) only occurs in the flat case.

\begin{remark}\label{remark:SDPD}
 We also note that the self-dual PD solution can be shown to be isometric to a two-centre Gibbons-Hawking metric with different masses (cf. \cite{Casteill}), which may suggest that, by adjusting the parameters, one could obtain the Eguchi-Hanson metric as a special case of \eqref{eq:PDmetric}. However, the analysis in this section shows that, if the PD metric \eqref{eq:PDmetric} is required to have Euclidean signature and to live in the manifold defined in section \ref{sec:PDcoord}, the metric \eqref{eq:PDmetric} does not contain the Eguchi-Hanson instanton.
\end{remark}

\section{The conformal Killing-Yano equation}
\label{appendix:CKY}

In this appendix we study the conformal Killing-Yano (CKY) equation on a four-dimensional Riemannian ALE manifold $(M,g)$, as we need this for Proposition \ref{prop:falloffOmega}. 

\subsection{Generalities}
We refer to Penrose and Rindler \cite{PR1, PR2} for background on the 2-spinor formalism. The CKY equation for a 2-form $Z_{ab}$ is 
\begin{align}\label{CKY}
L(Z)_{abc}:= \nabla_{a}Z_{bc} -\nabla_{[a}Z_{bc]} + 2g_{a[b}\xi_{c]} = 0
\end{align}
where $\xi_{a}=\frac{1}{3}\nabla^{b}Z_{ab}$. If $Z_{ab}$ is CKY, then its dual is also CKY, so we can assume $Z_{ab}$ to be self-dual. Then \eqref{CKY} adopts a simple form in spinor terms: writing $Z_{ab}=K_{AB}\epsilon_{A'B'}$ (with $K_{AB}=K_{BA}$), \eqref{CKY} is equivalent to the two-index twistor equation (or valence-2 Killing spinor equation):
\begin{align}\label{KSE}
\nabla_{A'(A}K_{BC)}=0.
\end{align}
These are $2\times 4=8$ real scalar equations. They can be conveniently written in the Geroch-Held-Penrose (GHP) formalism \cite[Section 4.12]{PR1}, adapted to Riemannian signature. We first choose an arbitrary unprimed dyad $(o^{A},\iota^{A})$, with $o_{A}\iota^{A}=1$ and\footnote{In Riemannian signature, spinor conjugation $\dagger$ maps any spinor $o^{A}$ to a linearly independent spinor $(o^{A})^{\dagger}$, is anti-linear, and satisfies $\dagger^2=-1$ (resp. $\dagger^2=+1$) on spinors with an odd (resp. even) number of indices. For example, $((o^{A})^{\dagger})^{\dagger}=-o^{A}$. The operation $\dagger$ extends to tensors, and in this case we denote it by the ordinary overbar $\bar{(\cdot)}$.} 
$\iota^{A}=(o^{A})^{\dagger}$, and write $K_0\equiv o^Ao^B K_{AB}$, $K_1\equiv i o^A\iota^B K_{AB}$, $K_2 \equiv \iota^{A}\iota^{B}K_{AB}$. Then $K_2=\overline{K}_0$ and $K_1=\overline{K}_1$. To construct a tetrad we also need a primed spin dyad $(\alpha^{A'},\beta^{A'})$, with $\alpha_{A'}\beta^{A'}=1$ and $\beta^{A'}=(\alpha^{A'})^{\dagger}$. Then a null tetrad $(\ell^{a},n^{a},m^{a},\tilde{m}^{a})$ is given by $\ell^{a}=o^{A}\alpha^{A'}$, $m^{a}=o^{A}\beta^{A'}$, $n^{a}=\overline{\ell}^{a}$ and $\tilde{m}^{a}=-\overline{m}^{a}$, and \eqref{KSE} becomes
\begin{equation}\label{KSEcomponents}
\begin{aligned}
& \text{\rm \th}K_0+2 i \kappa K_1 = 0, \qquad
(\text{\rm \th}'+2\rho')K_0 + 2 i(\text{\rm \dh}+\tau)K_1+2\sigma K_2 = 0, \\
& \text{\rm \dh}K_0+2 i \sigma K_1= 0, \qquad
 (\text{\rm \dh}'+2\tau')K_0 + 2 i(\text{\rm \th}+\rho)K_1+2\kappa K_2 = 0,
\end{aligned}
\end{equation}
plus the complex conjugates. In GHP language, $K_0,K_1,K_2$ have, respectively, weights $\{2,0\}$, $\{0,0\}$, $\{-2,0\}$. The action of $ \text{\rm \th}, \text{\rm \th}', \text{\rm \dh},  \text{\rm \dh}'$ on a scalar with weights $\{p,0\}$ is:
\begin{align}
 \text{\rm \th} = \ell^a\nabla_{a} - p\epsilon, \qquad
 \text{\rm \th}' = n^a\nabla_{a} + p\epsilon', \qquad
 \text{\rm \dh} = m^a\nabla_{a} - p\beta, \qquad
 \text{\rm \dh}' = \tilde{m}^a\nabla_{a} + p\beta'.
\end{align}
The spin coefficients are defined by 
\begin{equation}\label{spincoeff}
\begin{aligned}
 & \kappa = m^{a}\ell^{b}\nabla_{b}\ell_{a}, \qquad \sigma = m^{a}m^{b}\nabla_{b}\ell_{a}, \qquad \rho = m^{a}\tilde{m}^{b}\nabla_{b}\ell_{a}, \qquad \tau=m^{a}n^{b}\nabla_{b}\ell_{a}, \\
 & \epsilon = \tfrac{1}{2}(n^{a}\ell^{b}\nabla_{b}\ell_a+m^{a}\ell^{b}\nabla_{b}\tilde{m}_{a}), \qquad 
 \beta = \tfrac{1}{2}(n^{a}m^{b}\nabla_{b}\ell_a+m^{a}m^{b}\nabla_{b}\tilde{m}_{a}),
\end{aligned}
\end{equation}
and $\kappa'=-\bar\kappa$, $\sigma'=\bar\sigma$, $\rho'=\bar\rho$, $\tau'=-\bar\tau$, $\epsilon'=\bar\epsilon$, $\beta'=-\bar\beta$.
Once we solve \eqref{KSEcomponents} for $K_0,K_1,K_2$, the Killing spinor is $K_{AB} = K_0\iota_A\iota_B + 2 i K_1 o_{(A}\iota_{B)} + K_2 o_Ao_B$. To find the CKY tensor $Z_{ab}=K_{AB}\epsilon_{A'B'}$, we multiply by $\epsilon_{A'B'}$ and use the fact that a basis of real self-dual 2-forms is given by
\begin{equation}\label{SDforms}
\begin{aligned}
 \omega^{1}_{ab} ={}& -2 i o_{(A}\iota_{B)}\epsilon_{A'B} = (e^0\wedge e^1 + e^2\wedge e^3)_{ab}, \\ 
 \omega^{2}_{ab} ={}& i (o_Ao_B-\iota_A\iota_B)\epsilon_{A'B'} = (e^1\wedge e^2 + e^0\wedge e^3)_{ab}, \\ 
 \omega^{3}_{ab} ={}& -(o_Ao_B+\iota_A\iota_B)\epsilon_{A'B'} = (e^1\wedge e^3 - e^0\wedge e^2)_{ab},
\end{aligned}
\end{equation}
where we used the orthonormal coframe defined by $e^0=\frac{1}{\sqrt{2}}(\ell+n)$, $e^1=\frac{1}{i\sqrt{2}}(\ell-n)$, $e^2=\frac{1}{\sqrt{2}}(m-\tilde{m})$, $e^3=\frac{1}{i\sqrt{2}}(m+\tilde{m})$. Then the solution to \eqref{CKY} is
\begin{align}\label{CKYsolution}
Z_{ab} = a_1 \, \omega^{1}_{ab} + a_2 \, \omega^{2}_{ab} + a_3 \, \omega^{3}_{ab}
\end{align}
where 
\begin{align}\label{componentsCKY}
 a_1 = -K_1, \qquad a_2 = \tfrac{i}{2}(K_0-K_2), \qquad a_3 = -\tfrac{1}{2}(K_0+K_2).
\end{align}

\subsection{Euclidean 4-space}
Consider the case $(M,g)=(\mathbb{R}^4,g_0)$, where
\begin{align}
 g_{0} = \td{r}^2 + \frac{r^2}{4}\left[(\td\psi+\cos\theta\td\phi)^2+\td\theta^2+\sin^2\theta\td\phi^2 \right].
\end{align}
We choose the orthonormal coframe 
\begin{align}\label{ON}
e^0 = \tfrac{r}{2}(\td\psi+\cos\theta\td\phi), \qquad e^1=\td{r}, \qquad e^2=\tfrac{r}{2}\td\theta, \qquad e^3=\tfrac{r\sin\theta}{2}\td\phi, 
\end{align}
and the orientation $\varepsilon = e^0\wedge e^1\wedge e^2\wedge e^3$. We construct a null coframe $\ell = \frac{1}{\sqrt{2}}(e^0 + i e^1)$, $m=\frac{1}{\sqrt{2}}(e^2 + i e^3)$, $n=\bar\ell$, $\tilde{m}=-\bar{m}$. The associated spin coefficients \eqref{spincoeff} are found to be
\begin{align}
 \kappa=\sigma=\rho=\tau=0, \qquad \epsilon=\frac{i}{\sqrt{2} \; r}, \qquad \beta = -\frac{\cot\theta}{\sqrt{2} \; r}.
\end{align}
The condition $\kappa=\sigma=\rho=\tau=0$ automatically implies that $\omega^{1}$ in \eqref{SDforms} is K\"ahler.
We will focus on solutions to \eqref{CKY} which are invariant under $\mathfrak{t}^2={\rm span}(\partial_{\psi},\partial_{\phi})$, i.e. the functions $a_i$ in \eqref{componentsCKY} are $a_i=a_i(r,\theta)$. Then the left hand sides of \eqref{KSEcomponents} become
\begin{equation}\label{KSEcompE4}
\begin{aligned}
 \text{\rm \th}K_0+2 i\kappa K_1 ={}& \tfrac{i}{\sqrt{2}}(\partial_{r}-\tfrac{2}{r})K_0, \\
 \text{\rm \dh}K_0+2 i\sigma K_1={}& \tfrac{\sqrt{2}}{r}(\partial_{\theta}-\cot\theta)K_0, \\
 (\text{\rm \th}'+2\rho')K_0 + 2i(\text{\rm \dh}+\tau)K_1+2\sigma K_2 ={}& -\tfrac{i}{\sqrt{2}}(\partial_{r}+\tfrac{2}{r})K_0+\tfrac{i2\sqrt{2}}{r}\;\partial_{\theta}K_1, \\
 (\text{\rm \dh}'+2\tau')K_0 + 2i(\text{\rm \th}+\rho)K_1+2\kappa K_2 ={}& -\tfrac{\sqrt{2}}{r}(\partial_{\theta}+\cot\theta)K_0 - \sqrt{2} \; \partial_{r}K_1.
\end{aligned}
\end{equation}
Equating to zero, the solution is not difficult to find:
\begin{align}\label{Kflatspace}
 K_0 = -k_1 r^2\sin\theta = K_2, \qquad K_1 = k_1 r^2\cos\theta \red{-} k_2
\end{align}
where $k_1$ and $k_2$ are arbitrary real constants. Thus, from \eqref{componentsCKY} we get 
\begin{align}
a_1 = - k_1 r^2\cos\theta + k_2, \qquad a_2=0, \qquad a_3=k_1 r^2\sin\theta. 
\end{align}
Using \eqref{CKYsolution}, the general $\mathfrak{t}^2$-invariant CKY tensor in $\mathbb{E}^4$ is then
\begin{align}
 Z = k_1 \; r^2 \left( - \cos\theta \; \omega^1 +\sin\theta \; \omega^{3} \right) + k_2 \;\omega^{1},   \label{eq:Z0}
\end{align}
where $\omega^1,\omega^3$ are given by \eqref{SDforms} in terms of the coframe \eqref{ON}. We note that the solution $k_1=0$, $k_2\neq0$ is parallel (as noted before), thus the interesting solution is $k_1\neq0$, $k_2=0$. For generic $k_1,k_2$, the norm is $|Z|^2=4(k_1^2 r^4 + k^2_2 - 2 k_1 k_2 r^2 \cos \theta)$.

\subsection{ALE manifolds}
Let $(M,g)$ be ALE Ricci-flat, and let $Z$ be a solution to \eqref{CKY}. Let $Z^0$ is a solution to \eqref{CKY} in the asymptotic flat metric $\td r^2+r^2\gamma$, so from \eqref{eq:Z0} we deduce $Z^0=O(r^2)$. Writing $Z=Z^0+O(r)$ we will show that the subleading term is actually $O(r^{-2})$. From Definition \ref{def:ALE} and Remark \ref{remark:BKN}, the connection coefficients are $O(r^{-5})$, thus the CKY operator is $LZ=L^0Z+O(r^{-3})$, where $L^0$ is the CKY operator in flat space. We can then write the first equation in \eqref{KSEcompE4} as $r^2\partial_{r}(r^{-2}K_0)=O(r^{-3})$. The homogeneous solution gives \eqref{Kflatspace}, and the inhomogeneous solution gives $K_0=O(r^{-2})$. The rest of the equations in \eqref{KSEcompE4} give the inhomogenous solution $K_1=O(r^{-2})$ (and recall $K_2=\overline{K}_0$), so from \eqref{CKYsolution} and \eqref{componentsCKY} we see that $Z=Z^0+O(r^{-2})$.

\end{document}